\documentclass[11pt,pdftex]{amsart}

\usepackage{multirow}
\usepackage{amsmath,amsfonts,array,rotating,color,appendix,tabularx}

\usepackage[a4paper,lmargin=1.8cm,rmargin=1.8cm,tmargin=3.0cm,bmargin=3.0cm]{geometry}

\usepackage{tocvsec2}
\usepackage{booktabs}
\usepackage{float}
\usepackage{tikz}
\usetikzlibrary{positioning,arrows,shapes,decorations.pathmorphing,shadows}

\usepackage{hyperref}

\newcommand{\U}[1]{\mathrm{U}(#1)}
\newcommand{\liesp}[1]{\mathop{\mathfrak{sp}}(#1)}

\newcommand{\Spin}[1]{\ensuremath{\text{\upshape\rmfamily Spin}(#1)}}

\newcommand{\Sp}[1]{\mathrm{Sp}(#1)}
\newcommand{\SO}[1]{\mathrm{SO}(#1)}
\newcommand{\SU}[1]{\mathrm{SU}(#1)}

\newcommand{\RH}{R^\HH}
\newcommand{\LH}{L^\HH}
\newcommand{\RO}{R}

\newcommand{\liespin}[1]{\mathop{\mathfrak{spin}}(#1)}

\newcommand{\lieso}[1]{\mathop{\mathfrak{so}}(#1)}

\newcommand{\Id}{\mathop{\mathrm{Id}}}

\newcommand{\End}[1]{\mathrm{End}(#1)}
\newcommand{\Cl}[1]{\mathrm{Cl}_{#1}}

\newcommand{\FII}{\mathrm{FII}}
\newcommand{\EIII}{\mathrm{EIII}}
\newcommand{\EVI}{\mathrm{EVI}}
\newcommand{\EVIII}{\mathrm{EVIII}}

\newcommand{\CC}{\mathbb{C}}   
\newcommand{\HH}{\mathbb{H}}   
\newcommand{\RR}{\mathbb{R}}

\newcommand{\OO}{\mathbb{O}}

\numberwithin{equation}{section}

\newtheorem{theorem}{Theorem}[section]
\newtheorem*{te*}{Theorem}
\newtheorem{proposition}[theorem]{Proposition}

\theoremstyle{definition}
\newtheorem{definition}[theorem]{Definition}    

\theoremstyle{remark}
\newtheorem{remark}[theorem]{Remark}

\usepackage{amsmath,amsfonts,array,rotating}
\usepackage{hyperref}
\usepackage{float}
\newcommand{\tr}[1]{\mathop{\mathrm{tr}}(#1)}
\newcommand{\shortform}[1]{\ensuremath{{\scriptstyle\boldsymbol{#1}}}}

\numberwithin{equation}{section}

\begin{document}

\title{Clifford systems in octonionic geometry}  
%\titlerunning{$\Spin{9}$ and almost complex structures on $16$-dimensional manifolds}

%\date{\today}
\dedicatory{Dedicated to the memory of Sergio Console}

\renewcommand{\thefootnote}{\fnsymbol{footnote}}
\footnotetext[1]{All the authors were supported by the GNSAGA group of INdAM (the third author for a visit to Pisa and Rome in July 2015). The first and second authors were also supported by the MIUR under the PRIN Project ``Variet\`a reali e complesse: geometria, topologia e analisi armonica''. The third author was partially supported by CNCS UEFISCDI, project number PN-II-ID-PCE-2011-3-0118.}

\author{Maurizio Parton}
%~ \address{Universit\`a di Chieti-Pescara\\ Dipartimento di Economia, viale della Pineta 4, I-65129 Pescara, Italy}
%~ \email{parton@unich.it}
\author{Paolo Piccinni}
%~ \address{Sapienza-Universit\`a di Roma \\ Dipartimento di Matematica\\ 
%~ Piazzale Aldo Moro 2, I-00185, Roma, Italy 
%~ }
%~ \email{piccinni@mat.uniroma1.it}
\author{Victor Vuletescu$^*$}
%~ \address{University of Bucharest \\ Faculty of Mathematics and Informatics\\ 
%~ 14 Academiei str., 70109, Bucharest, Romania 
%~ }
%~ \email{vuli@fmi.unibuc.ro}

\begin{abstract}
We give an inductive construction for irreducible Clifford systems on Euclidean vector spaces. We then discuss how this notion can be adapted to Riemannian manifolds, and outline some developments in octonionic geometry.
\end{abstract}

\maketitle
\tableofcontents

%%\settocdepth{chapter}
%\begin{small}
%%\setlength\bibindent{4cm}
%\listoftables
%\end{small}
%%\settocdepth{section}

\section{Introduction}
The notion of Clifford system, as formalized in 1981 by D.~Ferus, H.~Karcher and H.~F.~M\"unzner, has been used in the last decades both in the study of isometric hypersurfaces and of Riemannian foliations \cite{fkm,r,gr}. In particular, Clifford systems have been used by Sergio Console and Carlos Olmos  \cite{co} to give an alternative proof of a Theorem of E. Cartan stating that a compact isoparametric hypersurface of a sphere with three distinct principal curvatures is a tube around the Veronese embedding of the projective planes $\RR P^2, \CC P^2, \HH P^2, \OO P^2$ over the reals, complex numbers, quaternions and Cayley numbers, respectively.

In this statement, the Veronese embedding of the four projective planes goes into spheres $S^4$, $S^7$, $S^{13}$, $S^{25}$ and these embeddings admit an analogy in complex projective geometry. Namely, the (so-called) projective planes 
\[
\CC P^2, \; (\CC\otimes\CC)P^2,\; (\CC\otimes\HH)P^2, \;  (\CC\otimes\OO)P^2
\]
over complex numbers and over the other three composition algebras of complex complex numbers, complex quaternions and complex octonions, admit an embedding into complex projective spaces $\CC P^5$, $\CC P^8$, $\CC P^{14}$, $\CC P^{26}$. These latter embeddings are also named after Veronese and give rise to projective algebraic varieties of degrees $4$, $6$, $14$, $78$, respectively. Very interesting properties of the mentioned two series of Veronese embeddings have been pointed out in \cite{ab}.

The following Table A collects ``projective planes'' $(\mathbb K \otimes \mathbb K')P^2$ over composition algebras $\mathbb K \otimes \mathbb K'$, where $\mathbb K,\mathbb K'\in\{\RR,\CC,\HH,\OO\}$. Here notations $V_2^4$, $V_4^6$, $V_8^{14}$, $V_{16}^{78}$ (with lower and upper indices being the complex dimension and the degree, respectively) are for the so-called \emph{Severi varieties}, smooth projective algebraic varieties with nice characterizations realizing the mentioned embeddings \cite{zak}. Table A will give a general reference for our discussion. In particular, the fourth Severi variety $\mathrm{E_6}/ \mathrm{Spin}(10) \cdot \mathrm{U}(1)  \cong V_{16}^{78} \subset \CC P^{26}$ has been recently studied both with respect to the structure given by its holonomy and in the representation of the differential forms that generate its cohomology \cite{pp3}. 

\renewcommand{\arraystretch}{1.65}
\begin{table}[h]
\caption{Projective planes}
\begin{center}
\resizebox*{1.00\textwidth}{!}{%
\begin{tabular}{|c||c|c|c|c|}
\hline
$_{\mathbb K \; =} \backslash ^{\mathbb K' \; =}$  & $\mathbb R$ &$\mathbb C$&$\mathbb H$&$\mathbb O$\\
\hline\hline
{$\RR$} &{$\RR P^2$}& $\CC P^2 \cong V_2^4$ & {
$\HH P^2 $}& {$\OO P^2 \cong \mathrm{F_4}/\Spin{9}$}\\
\hline
{$\CC$}&$\CC P^2 \cong V_2^4$ &$\CC P^2 \times \CC P^2 \cong V_4^6$ &$Gr_2(\CC^6) \cong V_8^{14}$ & $\mathrm{E_6}/ \mathrm{Spin}(10) \cdot \mathrm{U}(1)  \cong V_{16}^{78}$\\
\hline
{$\HH$} &{$\HH P^2$} & $Gr_2(\CC^6) \cong V_8^{14}$& {$Gr_4^{or}(\RR^{12})$} & {$\mathrm{E_7}/\mathrm{Spin}(12) \cdot \Sp{1} $}\\
\hline
{$\OO$} &{$\OO P^2 \cong \mathrm{F_4}/ \Spin{9}$} & $\mathrm{E_6}/ \mathrm{Spin}(10) \cdot \mathrm{U}(1) \cong V_{16}^{78}$&{$\mathrm{E_7}/ \mathrm{Spin}(12) \cdot \Sp{1}$}&{$\mathrm{E_8}/ \mathrm{Spin}(16)^+$} \\
\hline
\end{tabular}
}
\end{center}
\end{table}

Recall that a \emph{Clifford system} on the Euclidean vector space $\RR^N$ is the datum of an $(m+1)$-ple 
\[
C_m=(P_0,\dots , P_m)
\]
of symmetric transformations $P_\alpha$ such that:
\[
P_\alpha^2 = \; \Id \; \; \text{for all} \; \; \alpha, \qquad P_\alpha P_\beta = -P_\beta P_\alpha \; \; \text{for all} \; \; \alpha \neq \beta.
\]
A Clifford system on $\RR^N$ is said to be \emph{irreducible} if $\RR^N$ is not direct sum of two positive dimensional subspaces that are invariant under all the $P_\alpha$.

From representation theory of Clifford algebras one recognizes (cf.~\cite[page 483]{fkm}, \cite[page 163]{hu}) that $\RR^N$ admits an irreducible Clifford system $C=(P_0,\dots , P_m)$ if and only if
\[
N= 2\delta (m),
\]
where $\delta(m)$ is given by the following
\begin{table}[h]%\label{m}
\caption{Clifford systems}
\renewcommand{\arraystretch}{1.65}
\begin{center}
%\medskip
%\small{
\resizebox*{1.00\textwidth}{!}{%
\begin{tabular}{|c||c|c|c|c|c|c|c|c|c|c|c|c|c|c|c|c|c|c|c|c|c|}
\hline
$m$&$1$&$2$&$3$&$4$&$5$&$6$&$7$&$8$&$9$&$10$&$11$&$12$&$13$&$14$&$15$&$16$&\dots&$8+h$\\
\hline
$\delta(m)$&$1$&$2$&$4$&$4$&$8$&$8$&$8$&$8$&$16$&$32$&$64$&$64$&$128$&$128$&$128$&$128$&\dots&$16\delta(h)$\\
\hline
\end{tabular}
}
\end{center}
\end{table}

One can discuss uniqueness as follows. Given on $\RR^N$ two Clifford systems $C_m=(P_0,\dots , P_m)$ and $C'_m=(P'_0,\dots , P'_m)$, they are said to be \emph{equivalent} if there exists $A \in O(N)$ such that $P'_\alpha = A^t P_\alpha A$  for all $\alpha$. Then for $m \not\equiv 0$ mod $4$ there is a unique equivalence class of irreducible Clifford systems, and for $m \equiv 0$ mod $4$ there are two, classified by the two possible values of $\tr {P_0 P_1 \dots P_m} = \pm 2 \delta (m)$.

In the approach by Sergio Console and Carlos Olmos to the mentioned E.~Cartan theorem on isoparametric hypersurfaces with three distinct principal curvatures in spheres, the Clifford systems are related with the Weingarten operators of their focal manifolds, and the possible values of $m$ turn out to be here only $1,2,4$ or $8$, the multiplicities of the eigenvalues of the Weingarten operators. 

In the present paper we outline an inductive construction for all Clifford systems on real Euclidean vector spaces $\RR^N$, by pointing out how the four Clifford systems $C_1,C_2,C_4,C_8$ considered in \cite{co} correspond to structures given by the groups $\U{1}$, $\U{2}$, $\Sp{2}\cdot\Sp{1}$, $\Spin{9}$. We also develop, following ideas contained in \cite{pp,pp2,oppv,pp3}, the intermediate cases as well as some further cases appearing in Table B. We finally discuss the corresponding notion on Riemannian manifolds and relate it with the notion of even Clifford structure and with the octonionic geometry of some exceptional Riemannian symmetric spaces.

We just mentioned the even Clifford structures, a kind of unifying notion proposed by A.~Moroianu and U.~Semmelmann \cite{ms}. It is the datum, on a Riemannian manifold $M$, of a real oriented Euclidean vector bundle $(E,h)$, together with an algebra bundle morphism $\varphi:\text{Cl}^0(E) \rightarrow \End{TM}$ mapping $\Lambda^2 E$ into skew-symmetric endomorphisms. Indeed, a Clifford system gives rise to an even Clifford structure, but there are some even Clifford structures on manifolds that cannot be constructed, even locally, from Clifford systems. This will be illustrated by examples in Sections \ref{sectionEIII} and \ref{essential Clifford structures}.

\emph{Acknowledgements.} We thank F.\ Reese Harvey for his interest in the present work and for taking the reference \cite{dh} to our attention, cf.\ Remark \ref{dh}. We also thank the referee for pointing out how representation theory of Clifford algebras allows to give the basic Proposition \ref{basic} on essential even Clifford structures. Also proofs of Theorems \ref{EIII} and \ref{EVI} have been simplified through this approach.

\section{From $\RR$ to $\CC$ and to $\HH$: the Clifford systems $C_1,C_2,C_3, C_4$}\label{first four}

We examine here the first four columns of Table B, describing with some details irreducible Clifford systems $C_1,C_2.C_3,C_4$, acting on $\RR^2, \RR^4, \RR^8, \RR^8$, respectively.

A Clifford system $C_1$ on $\RR^2$ (here $m=1$ and $\delta (m) =1$) is given by matrices
\begin{equation*}%\label{eq:first}
N_0=\left(
\begin{array}{rr}
0 & 1 \\
1 & 0
\end{array}\right),\qquad
N_1=\left(
\begin{array}{rr}
1 & 0 \\
0 & -1
\end{array}\right),
\end{equation*}
representing in $\CC \cong \RR^2$ the involutions $z \in \CC \rightarrow i \bar z$ and $z \in \CC \rightarrow \bar z$, whose composition 
\[
N_{01} = N_0 N_1 = \left(
\begin{array}{rr}
0 & -1 \\
1 & 0
\end{array}\right)
\]
is  the complex structure on $\CC \cong \RR^2$. 

Going to the next case, the Clifford system $C_2$ (now $m=2$ and $\delta (m)=2$) is the prototype example of the Pauli matrices: 
\begin{equation*}%\label{eq:Ipauli}
P_0=\left(
\begin{array}{rr}
0 & 1 \\
1 & 0
\end{array}\right),\qquad
P_1=\left(
\begin{array}{rr}
0 & -i \\
i & 0
\end{array}\right),\qquad
P_2=\left(
\begin{array}{rr}
1 & 0 \\
0 & -1
\end{array}\right),
\end{equation*}
that we will need in their real representation:
\begin{equation*}%\label{eq:realpauli}
P_0=\left(
\begin{array}{r|r}
0 & \Id \\
\hline 
\Id & 0
\end{array}\right),\quad
P_1=\left(
\begin{array}{c|c}
0 & -N_{01} \\
\hline 
N_{01}& 0
\end{array}\right),\qquad
P_2=\left(
\begin{array}{r|r}
\Id & 0 \\
\hline
0 & -\Id
\end{array}\right).
\end{equation*}

The compositions $P_{\alpha\beta} = P_\alpha P_\beta$, for $\alpha<\beta$, yield as complex structures on $\RR^4$
the multiplication on the right $\RH_i, \RH_j, \RH_k$ by unit quaternions $i,j,k$: 
\[
\resizebox*{1.00\textwidth}{!}{%
$
%\begin{split}
P_{01}=\RH_i =
%\scriptsize{
\left(
\begin{array}{rrrr}
0 & -1 & 0 & 0 \\
1 & 0 & 0 & 0 \\
0 & 0 & 0 & 1 \\
0 & 0 & -1 & 0
\end{array}
\right)
%}
,\enskip
P_{02}	=\RH_j =
%\scriptsize{
\left(
\begin{array}{rrrr}
0 & 0 & -1 & 0 \\
0 & 0 & 0 & -1 \\
1 & 0 & 0 & 0 \\
0 & 1 & 0 & 0
\end{array}
\right)
%}
,\enskip
P_{12}=\RH_k =
%\scriptsize{
\left(
\begin{array}{rrrr}
0 & 0 & 0 & -1 \\
0 & 0 & 1 & 0 \\
0 & -1 & 0 & 0 \\
1 & 0 & 0 & 0
\end{array}
\right)
%}
.
%\end{split}
$
}
\]
Multiplication $\LH_i$ on the left by $i$ coincides with 
\[
P_{012} = P_0 P_1 P_2 = \LH_i =
%\scriptsize{
\left(
\begin{array}{rrrr}
0 & -1 & 0 & 0 \\
1 & 0 & 0 & 0 \\
0 & 0 & 0 & -1 \\
0 & 0 & 1 & 0
\end{array}
\right)
%}
,
\] 
and to complete $\RH_i, \RH_j,\RH_j,\LH_i$ to a basis of the Lie algebra $\lieso{4} \cong \liesp{1} \oplus \liesp{1}$ one has to add the two further left multiplications 
\[
\LH_j =
%\scriptsize{
\left(
\begin{array}{rrrr}
0 & 0 & -1 & 0 \\
0 & 0 & 0 & 1 \\
1 & 0 & 0 & 0 \\
0 & -1 & 0 & 0
\end{array}
\right)
%}
,\qquad
\LH_k =
%\scriptsize{
\left(
\begin{array}{rrrr}
0 & 0 & 0 & -1 \\
0 & 0 & -1 & 0 \\
0 & 1 & 0 & 0 \\
1 & 0 & 0 & 0
\end{array}
\right)
%}
.
\] 
Thus:
\begin{proposition}\label{C_2} 
Orthogonal linear transfomations in $\RR^4$ preserving the individual $P_0,P_1,P_2$ are the ones in $\U{1}=\SO{2}_\Delta \subset \SO{4}$, those preserving the vector space $E^3=<P_0,P_1,P_2> $ are the ones in $\U{2}=\Sp{1}\cdot \U{1}$.
\end{proposition}

%Matrices in $\SO{4}$ that preserve our datum $V^3  \subset \End{\RR^4}$  are the ones in $\mathrm U(2)= \Sp{1} \cdot \mathrm{U}(1)$.

The next Clifford systems $C_3$ and $C_4$ act on $\RR^8$. They can be defined by the following $4\times 4$ block matrices
\begin{equation*}%\label{eq:SO(4)}
\begin{split}
C_3: \qquad Q'_0&=\left(
\begin{array}{r|r}
0 & \Id \\
\hline 
\Id & 0
\end{array}\right),\qquad
Q'_1=\left(
\begin{array}{c|c}
0 & -P_{01} \\
\hline 
P_{01}& 0
\end{array}\right),\\
Q'_2&=\left(
\begin{array}{c|c}
0 & -P_{02} \\
\hline 
P_{02}& 0
\end{array}\right),\qquad
Q'_3=\left(
\begin{array}{r|r}
\Id & 0 \\
\hline
0 & -\Id
\end{array}\right),
\end{split}
\end{equation*}
and
\begin{equation*}%\label{eq:SO(4)}
C_4: \qquad  Q_0 =Q'_0,\enskip
Q_1 =Q'_1,\enskip
Q_2 =Q'_2,\enskip
Q_3=\left(
\begin{array}{c|c}
0 & -P_{12} \\
\hline 
P_{12}& 0
\end{array}\right),\enskip
Q_4= Q'_3.
\end{equation*}

The following characterizations of the structures on $\RR^8$ associated with $C_3$ and $C_4$ are easily seen.

\begin{proposition}\label{C_3} 
The structure defined in $\RR^8$  by the datum of the vector space $E^4 = <Q'_0,Q'_1,Q'_2,Q'_3>\subset\End{\RR^8}$ can be described as follows.
Matrices
\[
B=\left(
\begin{array}{c|c}
B' & B'' \\ 
\hline
B''' & B''''
\end{array} \right)
\]
commuting with the single endomorphisms $Q'_0,Q'_1,Q'_2,Q'_3 $ are characterized by the conditions $B'=B'''' \in \Sp{1} \subset \SO{4}$ and $B''=B'''=0$. Accordingly, matrices of $\SO{8}$ that preserve the vector space $E^4$ belong to a subgroup $\Sp{1} \cdot \Sp{1} \cdot \Sp{1} \subset \SO{8}$.
\end{proposition}

To prepare the next Clifford systems, namely $C_5,C_6,C_7,C_8$ on $\RR^{16}$, we need to look at the complex structures $Q_{0 \alpha}= Q_0Q_\alpha$ on $\RR^8$. They indeed coincide with $R_i,R_j,R_k,R_e$, the right multiplication on $\RR^8 \cong \OO$ by octonions $i,j,k,e$, respectively:

\begin{equation*}%\label{pauliquatoct}
\begin{split}
Q_{01}=R_i&=\left(
\begin{array}{c|c}
\RH_i & 0 \\ \hline
0 & -\RH_i
\end{array}
\right),\qquad
Q_{02}=R_j=\left(
\begin{array}{c|c}
\RH_j & 0 \\ \hline
0 & -\RH_j
\end{array}
\right),\\
Q_{03}=R_k&=\left(
\begin{array}{c|c}
\RH_{k}& 0 \\ \hline
0& -\RH_{k}
\end{array}
\right),\qquad
Q_{04}=R_e=\left(
\begin{array}{c|c}
0& -\Id \\ \hline
\Id&0
\end{array}
\right).
\end{split}
\end{equation*}

Associated with the Clifford system $C_4=(Q_0,Q_1,Q_2,Q_3,Q_4)$, we have ten complex structures $Q_{\alpha\beta}=Q_\alpha Q_\beta$ with $\alpha < \beta$, a basis of the Lie algebra $\liesp{2} \subset \lieso{8}$. 

Their K\"ahler forms $\theta_{\alpha\beta}$, written in the coordinates of $\RR^8$, and using short notations like $\shortform{12} = dx_1 \wedge dx_2$, read:
\begin{equation*}%\label{eq:theta1}
\begin{aligned}
\theta_{01} &= -\shortform{12}+\shortform{34}+\shortform{56}-\shortform{78}, &
\theta_{02} &= -\shortform{13}-\shortform{24}+\shortform{57}+\shortform{68}, &
\theta_{03} &= -\shortform{14}+\shortform{23}+\shortform{58}-\shortform{67}, \\
\theta_{12} &= -\shortform{14}+\shortform{23}-\shortform{58}+\shortform{67}, &
\theta_{13} &= \shortform{13}+\shortform{24}+\shortform{57}+\shortform{68}, &
\theta_{23} &= -\shortform{12}+\shortform{34}-\shortform{56}+\shortform{78}, \\
\theta_{04} &= -\shortform{15}-\shortform{26}-\shortform{37}-\shortform{48}, & 
\theta_{14} &= -\shortform{16}+\shortform{25}+\shortform{38}-\shortform{47}, & \\
\theta_{24} &= -\shortform{17}-\shortform{28}+\shortform{35}+\shortform{46}, & 
\theta_{34} &= -\shortform{18}+\shortform{27}-\shortform{36}+\shortform{45}.
\end{aligned}
\end{equation*}
Thus, the second coefficient of the characteristic polynomial of their skew-symmetric matrix $\theta =(\theta_{\alpha\beta})$ turns out to be:
\begin{equation*}%\label{Theta}
\tau_2(\theta)=\sum_{\alpha < \beta} \theta^2_{\alpha \beta} =
-12\shortform{1234}-4\shortform{1256}-4\shortform{1357}+4\shortform{1368}-4\shortform{1278}-4\shortform{1467}-4\shortform{1458}+\star = -2 \Omega_L,
\end{equation*}
where $\star$ denotes the Hodge star of what appears before it, and where
\[
\Omega_L = \omega^2_{\LH_i}+\omega^2_{\LH_j}+\omega^2_{\LH_k}
\]
is the left quaternionic 4-form. 

One can check that matrices 
$
B=\left(
\begin{array}{c|c}
B' & B'' \\ 
\hline 
B''' & B'''' 
\end{array}
\right)\in\SO{8}
$ commuting with each of the $Q_\alpha$ are again the ones satisfying $B''=B'''=0$ and $B'=B''''\in\Sp{1}\subset\SO{4}$. Hence the stabilizer of all individual $Q_\alpha$ is the diagonal $\Sp{1}_\Delta \subset\SO{8}$, and the stabilizer of their spanned vector space $E^5$ is the quaternionic group $\Sp{2}\cdot\Sp{1} \subset \SO{8}$. Thus:

\begin{proposition}\label{C_4}
The vector space $E^5$ spanned by the Clifford system $C_4=(Q_0,Q_1,Q_2,Q_3,Q_4)$ gives rise to the quaternion Hermitian structure of $\RR^8$, and it is therefore equivalent to the datum either of the reduction to $\Sp{2}\cdot\Sp{1} \subset \SO{8}$ of its structure group, or to the quaternionic 4-form $\Omega_L$. The complex structures $Q_{\alpha \beta} = Q_\alpha Q_\beta$ are for $\alpha < \beta$ a basis of the Lie subalgebra $\mathfrak{sp}(2) \subset \mathfrak{so}(8)$.
\end{proposition}

\begin{remark}\label{two}
As mentioned in the Introduction, when $m \equiv 0$ mod $4$, there are two equivalence classes of Clifford systems. It is clear from the construction of $C_4$ that a representative of the other class is just $\tilde C_4=(Q_0, \tilde Q_1, \tilde Q_2, \tilde Q_3, Q_4)$, where:
\begin{equation*}
\tilde Q_1=\left(
\begin{array}{c|c}
0 & -\LH_i \\ \hline
\LH_i & 0
\end{array}
\right),\qquad
\tilde Q_2=\left(
\begin{array}{c|c}
0 & -\LH_j \\ \hline
\LH_j & 0
\end{array}
\right),\qquad 
\tilde Q_3=\left(
\begin{array}{c|c}
0 & -\LH_k \\ \hline
\LH_k & 0
\end{array}
\right). 
\end{equation*}
\end{remark}

\section{Statement on how to write new Clifford systems and representation theory}\label{Statements}

The Clifford systems $C_3$ and $C_4$ have been obtained from $C_2$ through the following procedure. Similarly for the step $C_1 \rightarrow C_2$. 

\begin{theorem}\label{Procedure}{\rm[Procedure to write new Clifford systems from old]} Let $C_m= (P_0, P_1, \dots , P_m)$ be the last (or unique) Clifford system in $\RR^N$. Then the first (or unique) Clifford system \[
C_{m+1}=(Q_0, Q_1, \dots , Q_m, Q_{m+1})
\]  
in $\RR^{2N}$ has as first and as last endomorphisms respectively
\[
Q_0 =\left(
\begin{array}{r|r}
0 & \Id \\
\hline 
\Id & 0
\end{array}\right),\qquad Q_{m+1}=\left(
\begin{array}{r|r}
\Id & 0\\
\hline 
 0 & -\Id
\end{array}\right),
\]
where the blocks are $N \times N$. The remaining matrices are
\[
Q_\alpha =\left(
\begin{array}{c|c}
0 & -P_{0\alpha} \\
\hline 
P_{0 \alpha} & 0
\end{array}\right)\qquad \alpha=1,\dots,m. 
\]
Here $P_{0 \alpha}$ are the complex structures given by compositions $P_0 P_\alpha$ in the Clifford system $C_m$. When the complex structures $P_{0\alpha}$ can be viewed as (possibly block-wise) right multiplications by some of the unit quaternions $i,j,k$ or unit octonions $i,j,k,e,f,g,h$, and if the dimension permits, similarly defined further endomorphisms $Q_\beta$ can be added by using some others among $i,j,k$ or $i,j,k,e,f,g,h$.
\end{theorem}

\begin{proof} Since $P_{0\alpha}P_{0\alpha}=- \Id$, it is straightforward that $C_{m+1}$ is a Clifford system. As for the statement concerning the further $Q_\beta$, its proof follows as in the cases of $C_4$ (already seen), the further cases of $C_6,C_7,C_8$ in the next Section, and of $C_{12} ,C_{14,},C_{15},C_{16}$ in Section \ref{other 7}.
\end{proof}

%\begin{remark} As observed in \cite[pages 482--483]{fkm}, any irreducible Clifford system $C_m$ in $\RR^{2\delta (m)}$ gives rise to an irreducible representation of the real Clifford algebra $\Cl{m-1}$ in $\RR^{\delta(m)}$, and vice versa. Thus the former procedure to write Clifford systems can be seen as rephrasing the way to get irreducible representations of Clifford algebras. For these latter one can see \cite[pages 30--40]{lm}, and more explicitly the construction of Clifford algebra representations in \cite[pages 18--20]{tr}.
%\end{remark}
We now discuss some aspects of Clifford systems and of even Clifford structures (defined in the Introduction) related with representation theory of Clifford algebras. As pointed out in in \cite[pages 482--483]{fkm}, any irreducible Clifford system $C_m=(P_0, \dots , P_m)$ in $\RR^N,  \, N=2\delta (m)$, gives rise to an irreducible representation of the Clifford algebra $\Cl{0,m-1}$ in $\RR^{\delta(m)}$. This latter is given by skew-symmetric  matrices
\[
E_1, \dots , E_{m-1} \in \mathfrak{so} (\delta (m))
\]
satisfying
\[
E_\alpha E_\beta +E_\beta E_\alpha =-2\delta_{\alpha \beta} I.
\]
To get such matrices $E_\alpha$ from $C_m$ consider the $\delta (m)$-dimensional subspace $E_+$ of $\RR^N$ defined as the $(+1)$-eigenspace  of $P_0$, and observe that $E_+$ is also invariant under the compositions $P_1P_{\alpha +1}$. Then define the skew-symmetric endomorphisms on $\RR^{\delta (m)}$
\[
E_\alpha =P_\alpha P_{\alpha +1} \vert_{E_+} \quad (\alpha =1,\dots , m-1).
\]
This gives the system of $E_1, \dots , E_{m-1} \in \mathfrak{so}(\delta (m))$, thus a representation of the Clifford algebra $\Cl{0,m-1}$ in $\RR^{\delta(m)}$. Conversely, given the anti-commuting $E_1, \dots , E_{m-1} \in \mathfrak{so}(\delta (m))$, define on $\RR^N$ the Clifford system $C_m$ given by the symmetric involutions
\[
P_0(u,v)=(v,u), \dots , P_\alpha (u,v)=(-E_\alpha v, E_\alpha u), \dots , P_m(u,v)=(u,-v).
\]

As a consequence, the procedure given by Theorem \ref{Procedure} can be seen as rephrasing the way to get irreducible representations of Clifford algebras. For these latter one can see \cite[pages 30--40]{lm}, and more explicitly the construction of Clifford algebra representations in \cite[pages 18--20]{tr}.

\begin{remark} An alternative aspect of Clifford systems is to look at $C_m$ in $\RR^N, \, N=2 \delta (m)$, as a representation of the Clifford algebra $\Cl{0,m+1}$ in $\RR^N$ such that any vector of the pseudo-euclidean $\RR^{0,m+1} \subset \Cl{0,m+1}$ acts as a symmetric endomorphism in $\RR^N$.
\end{remark}

Recall now from the structure of Clifford algebras the following periodicity relations
\[
\Cl{m+8} \cong \Cl{m} \otimes \RR (16), \qquad \Cl{0, m+8} \cong \Cl{0,m} \otimes \RR (16),
\]
where $\RR(16)$ denotes the algebras of all real matrices of order $16$. 

%Also, for low values of $m$ one has the table

%\begin{table}[h]\label{m-1}
%\caption{Representations of $\Cl{0,m-1}$ in $\RR^{\delta (m)}$}
%\renewcommand{\arraystretch}{1.25}
%\begin{center}
%%\medskip
%%\small{
%\resizebox*{0.90\textwidth}{!}{%
%\begin{tabular}{|c||c|c|c|c|c|c|c|c|c|c|c|c|}
%\hline
%$m$&$1$&$2$&$3$&$4$&$5$&$6$&$7$&$8$&$9$\\
%\hline
%$\delta(m)$&$1$&$2$&$4$&$4$&$8$&$8$&$8$&$8$&$16$\\
%\hline
%$\Cl{0,m-1}$&$\RR$&$\RR \oplus \RR$&$\RR (2)$&$\CC (2) $&$\HH (2)$&$\HH (2) \oplus \HH (2)$&$\HH (4)$&$ \CC (8)$&$\RR (16)$\\
%\hline
%\end{tabular}
%}
%\end{center}
%\end{table}

%\noindent so that the irreducible representation of $\Cl{m-1}$ in $\RR^{\delta (m)}$ is faithful unless $m-1 \equiv 3,7$ mod $4$.

Look now at the even Clifford structures, mentioned in the Introduction. First observe that a natural notion of irreducibility can be given for them, by requiring the Euclidean vector bundle $(E,h)$ not to be a direct sum. Then, by definition an irreducible even Clifford structure of rank $m+1$ is equivalent to an irreducible representation of the even Clifford algebra $\Cl{m+1}^0 \cong \Cl{m}$ in $\RR^N, \, N=2\delta (m)$, mapping $\Lambda^2 \RR^{m+1}$ into skew-symmetric endomorphisms of $\RR^N$. 

The mentioned representations are listed, for low values of $m$, in the following:

\begin{table}[h]\label{m-1}
\caption{Representations of $\Cl{0,m-1}$ in $\RR^{\delta (m)}$ and of $\Cl{m+1}^0$ in $\RR^{2\delta (m)}$}
\renewcommand{\arraystretch}{1.35}
\begin{center}
%\medskip
%\small{
\resizebox*{1.00\textwidth}{!}{
\begin{tabular}{|c||c|c|c|c|c|c|c|c|c|c|c|c|}
\hline
$m$&$1$&$2$&$3$&$4$&$5$&$6$&$7$&$8$&$9$\\
\hline
\hline
\hline
$\delta (m)$&$1$&$2$&$4$&$4$&$8$&$8$&$8$&$8$&$16$\\
\hline
$\Cl{0,m-1}$&$\RR$&$\RR \oplus \RR$&$\RR (2)$&$\CC (2) $&$\HH (2)$&$\HH (2) \oplus \HH (2)$&$\HH (4)$&$ \CC (8)$&$\RR (16)$\\
\hline
\hline
\hline
$2\delta (m)$&$2$&$4$&$8$&$8$&$16$&$16$&$16$&$16$&$32$\\
\hline
$\Cl{m+1}^0 \cong \Cl{m}$&$\CC$&$\HH$&$\HH \oplus \HH$&$\HH (2)$&$\CC (4)$&$\RR (8)$&$\RR (8) \oplus \RR (8)$&$\RR (16)$&$\CC (16)$\\
\hline
\end{tabular}
}
\end{center}
\end{table}

Of course, a Clifford system $C_m=(P_0, \dots , P_m)$ in $\RR^N$ gives rise to an even Clifford structure on the same $\RR^N$ just by requiring the vector bundle $E$ to be the vector sub-bundle of the endomorphism bundle generated by $P_0, \dots , P_m$. Not every irreducible even Clifford structure can be obtained in this way, and not only by dimensional reasons, as we will see on manifolds, cf. Section \ref{sectionEIII}. We call \emph{essential} an irreducible even Clifford structure that is not induced by an irreducible Clifford system. Thus, to see whether an irreducible even Clifford structure is essential, Table C and the mentioned periodicity relation can be used. This gives the following:

\begin{proposition}\label{basic}
Irreducible even Clifford structures of rank $m+1$ on $\RR^{2\delta (m)}$ are essential when $m \equiv 3,5,6,7$ mod. $8$, and non essential when $m=0,4$ mod $8$. For $m=1,2$ mod $8$ both possibilities are open.
\end{proposition}

We will can back to this point on manifolds, see the last Sections.

\section{From $\HH$ to $\OO$: the Clifford systems $C_5,C_6,C_7,C_8$}\label{R16}

According to Table B and to Theorem \ref{Procedure}, the next Clifford system to consider is 
\[
C_5=(S'_0,S'_1,S'_2,S'_3,S'_4,S'_5)
\] 
in $\RR^{16}$, where:
\begin{equation*}%\label{eq:SO(4)}
\begin{split}
S'_0&=\left(
\begin{array}{r|r}
0 & \Id \\
\hline 
\Id & 0
\end{array}\right),\qquad
S'_1=\left(
\begin{array}{c|c}
0 & -Q_{01} \\
\hline 
Q_{01}& 0
\end{array}\right),\qquad
S'_2=\left(
\begin{array}{c|c}
0 & -Q_{02} \\
\hline 
Q_{02}& 0
\end{array}\right),\\
S'_3&=\left(
\begin{array}{c|c}
0 & -Q_{03} \\
\hline 
Q_{03}& 0
\end{array}\right),\qquad
S'_4=\left(
\begin{array}{c|c}
0 & -Q_{04} \\
\hline 
Q_{04}& 0
\end{array}\right),\qquad
S'_5=\left(
\begin{array}{r|r}
\Id & 0 \\
\hline
0 & -\Id
\end{array}\right).
\end{split}
\end{equation*}

A computation shows that:

\begin{proposition}\label{C_5}
The orthogonal transformations in $\RR^{16}$ commuting with the individual $S'_0, \dots , S'_5$ are the ones in the diagonal $\Sp{1}_\Delta \subset \SO{16}.$ The orthogonal transformations preserving the vector subspace $E^6 = <C_5> \subset \End{\RR^{16}}$ are the ones in the subgroup $\SU{4} \cdot \Sp{1} \subset \SO{16}$.The complex structures $S'_{\alpha \beta} = S'_\alpha S'_\beta$ are for $\alpha < \beta$ a basis of a Lie subalgebra $\mathfrak{su}(4) \subset \mathfrak{so}(16)$.
\end{proposition}

By reminding that

\begin{equation*}%\label{rightoctonions}
Q_{01}=R_i,\enskip
Q_{02}=R_j,\enskip
%J_{12}=\left(
%\begin{array}{c|c}
%0 &\RH_j \\ \hline
%\RH_j & 0
%\end{array}
%\right) = -R_{je},
%\end{equation} 
%\begin{equation*}
Q_{03}= R_k,\enskip
Q_{04}= R_e,\enskip
%J_{13}=\left(
%\begin{array}{c|c}
%\RH_i& 0\\ \hline
%0& \RH_i
%\end{array}
%\right) = -R_{jk} , \;
%J_{23}=\left(
%\begin{array}{c|c}
%0& -\RH_k\\ \hline
%-\RH_k &0
%\end{array}
%\right) =R_{ke}. 
\end{equation*}
are the right multiplications on $\OO$ by $i,j,k,e$, one completes $C_5$ to the Clifford system 
\[
C_8=(S_0,S_1,S_2,S_3,S_4,S_5,S_6,S_7,S_8)
\]
with
\[
S_0=S'_0,\enskip
S_1=S'_1,\enskip
S_2=S'_2,\enskip
S_3=S'_3,\enskip
S_4=S'_4,\enskip
S_8=S'_5
\]
and
\begin{equation*}%\label{eq:SO(4)}
S_5=\left(
\begin{array}{r|r}
0 & -R_f \\
\hline 
R_f & 0
\end{array}\right),\qquad
S_6=\left(
\begin{array}{c|c}
0 & -R_g \\
\hline 
R_g& 0
\end{array}\right),\qquad
S_7=\left(
\begin{array}{c|c}
0 & -R_h \\
\hline 
R_h& 0
\end{array}\right).
\end{equation*}

It is now natural to compare the Clifford system $C_8$ with the following notion, that was proposed by Th.~Friedrich in \cite{fr}.
\begin{definition}\label{defspin9}
A $\Spin{9}$-\emph{structure} on a 16-dimensional Riemannian manifold $(M,g)$ is a rank $9$ real vector bundle 
\[
E^9\subset\End{TM}\rightarrow M,
\] 
locally spanned by self-dual anti-commuting involutions $S_{\alpha}: TM \rightarrow TM$:  
\begin{equation*}%\label{1}
\begin{split}
g(S_{\alpha} X,Y)&=g(X,S_{\alpha}Y), \qquad \alpha=1,\dots,9,\\
S_{\alpha} S_{\beta}&=-S_{\beta} S_{\alpha}, \qquad \alpha\neq\beta,\\
S_{\alpha}^2&=\Id.
\end{split}
\end{equation*}
\end{definition}

From this datum one gets on $M$ local almost complex structures 
$S_{\alpha\beta}=S_{\alpha} S_{\beta}$,
and the $9\times 9$ skew-symmetric matrix of their K\"ahler $2$-forms 
\begin{equation*}%\label{kae}
\psi = (\psi_{\alpha\beta})
\end{equation*}
is naturally associated. The differential forms $\psi_{\alpha\beta}$, $\alpha<\beta$, are thus \emph{a local system of K\"ahler $2$-forms} on the $\Spin{9}$-manifold $(M^{16},E^9)$.

On the model space $\RR^{16}$, the standard $\Spin{9}$-structure is defined by the generators $S_1, \dots, S_9$ of the Clifford algebra $\Cl{9}$, the endomorphisms' algebra of its $16$-dimensional real representation $\Delta_9 = \RR^{16} = \OO^2$.
% Definition~\ref{de:spin9structure} relates then to the following characterization of $\Spin{9}$. 
Accordingly, unit vectors in $\RR^9$ can be viewed as self-dual endomorphisms 
\[
w: \Delta_9 = \OO^2 \rightarrow \Delta_9 = \OO^2,
\] 
and the action of $w = u + r \in S^8$ ($u \in \OO$, $r \in \RR$, $u\overline u + r^2 =1$), on pairs $(x,x') \in \OO^2$ is given by
\begin{equation}\label{HarSpC}
\left(
\begin{array}{c}
x \\
x'
\end{array}
\right)
\longrightarrow
\left(
\begin{array}{cc}
r & R_{\overline u} \\ 
R_u & -r
\end{array}
\right)  
\left(
\begin{array}{c} 
x \\ 
x'
\end{array}
\right),
\end{equation}
where $R_u, R_{\overline u}$ denote the right multiplications by $u, \overline u$, respectively  (cf.~\cite[page 288]{ha}). The choices
\[
w = (0,1),(0,i),(0,j),(0,k),(0,e),(0,f),(0,g),(0,h)\text{ and }(1,0) \in S^8 \subset \OO \times \RR = \RR^9,
\]
define the symmetric endomorphisms:
\begin{equation*}%\label{top}
S_0, S_1, \dots , S_8
\end{equation*}
that constitute our Clifford system $C_8$.

The subgroup $\Spin{9} \subset \SO{16}$ is characterized as preserving the vector subspace
\begin{equation*}%\label{eq:V9}
E^9 = <S_0,\dots,S_8>\subset \End{\RR^{16}},
\end{equation*}
whereas it is easy to check that the only matrices in $\SO{16}$ that preserve all the individual involutions $S_0, \dots , S_8$ are just $\pm \Id$.

It is useful to have explicitly the following right multiplications on $\OO$:

\begin{equation*}%\label{matricesspin7}
\begin{split}
\RO_i=\left(
\begin{array}{c|c}
\RH_i & 0 \\ \hline
0 & -\RH_i
\end{array}
\right),\qquad
\RO_j&=\left(
\begin{array}{c|c}
\RH_j & 0 \\ \hline
0 & -\RH_j
\end{array}
\right),\qquad
\RO_k=\left(
\begin{array}{c|c}
\RH_k & 0 \\ \hline
0 & -\RH_k
\end{array}
\right),\\
\RO_e&=\left(
\begin{array}{c|c}
0 & -\Id \\ \hline
\Id & 0
\end{array}
\right),\\
\RO_f=\left(
\begin{array}{c|c}
0 & \LH_i \\ \hline
\LH_i & 0
\end{array}
\right),\qquad
\RO_g&=\left(
\begin{array}{c|c}
0 & \LH_j \\ \hline
\LH_j & 0
\end{array}
\right),\qquad
\RO_h=\left(
\begin{array}{c|c}
0 & \LH_k \\ \hline
\LH_k & 0
\end{array}
\right).
\end{split}
\end{equation*}

The space $\Lambda^2\RR^{16}$ of $2$-forms in $\RR^{16}$ decomposes under $\Spin{9}$ as ~\cite[page 146]{fr}: 
\begin{equation}\label{decomposition}
\Lambda^2\RR^{16} = \Lambda^2_{36} \oplus \Lambda^2_{84}
\end{equation}
where  $\Lambda^2_{36} \cong \liespin{9}$ and $ \Lambda^2_{84}$ is an orthogonal complement in $\Lambda^2 \cong \lieso{16}$. Bases of the two subspaces are given respectively by 
\[
S_{\alpha \beta} = S_\alpha S_\beta\qquad\text{if }\alpha <\beta \qquad \text{and} \qquad
S_{\alpha \beta \gamma} = S_\alpha S_\beta S_\gamma \qquad \text{if }\alpha <\beta<\gamma,
\] 
all complex structures on $\RR^{16}$.
We will need for later use the following ones:
%~ \begin{equation*}%\label{eq:J1}
%~ \begin{aligned}
%~ S_{01}&=\left(
%~ \begin{array}{c|c}
%~ \RO_i & 0 \\ \hline
%~ 0 & -\RO_i
%~ \end{array}
%~ \right),&
%~ S_{02}&=\left(
%~ \begin{array}{c|c}
%~ \RO_j & 0 \\ \hline
%~ 0 & -\RO_j
%~ \end{array}
%~ \right),&
%~ S_{03}&=\left(
%~ \begin{array}{c|c}
%~ \RO_k & 0 \\ \hline
%~ 0 & -\RO_k
%~ \end{array}
%~ \right),&
%~ S_{04}&=\left(
%~ \begin{array}{c|c}
%~ \RO_e & 0 \\ \hline
%~ 0 & -\RO_e
%~ \end{array}
%~ \right),
%~ \\
%~ S_{05}&=\left(
%~ \begin{array}{c|c}
%~ \RO_f & 0 \\ \hline
%~ 0 & -\RO_f
%~ \end{array}
%~ \right),&
%~ S_{06}&=\left(
%~ \begin{array}{c|c}
%~ \RO_g & 0 \\ \hline
%~ 0 & -\RO_g
%~ \end{array}
%~ \right),&
%~ S_{07}&=\left(
%~ \begin{array}{c|c}
%~ \RO_h & 0 \\ \hline
%~ 0 & -\RO_h
%~ \end{array}
%~ \right),&
%~ S_{08}&=\left(
%~ \begin{array}{c|c}
%~ 0 & -\Id \\ \hline
%~ \Id & 0
%~ \end{array}
%~ \right),&
%~ \end{aligned}
%~ \end{equation*}
\[
\resizebox*{1.00\textwidth}{!}{%
$
S_{01}=\left(
\begin{array}{c|c}
\RO_i & 0 \\ \hline
0 & -\RO_i
\end{array}
\right),\enskip
S_{02}=\left(
\begin{array}{c|c}
\RO_j & 0 \\ \hline
0 & -\RO_j
\end{array}
\right),\enskip
S_{03}=\left(
\begin{array}{c|c}
\RO_k & 0 \\ \hline
0 & -\RO_k
\end{array}
\right),\enskip
S_{04}=\left(
\begin{array}{c|c}
\RO_e & 0 \\ \hline
0 & -\RO_e
\end{array}
\right),
$
}
\]
\[
\resizebox*{1.00\textwidth}{!}{%
$
S_{05}=\left(
\begin{array}{c|c}
\RO_f & 0 \\ \hline
0 & -\RO_f
\end{array}
\right),\enskip
S_{06}=\left(
\begin{array}{c|c}
\RO_g & 0 \\ \hline
0 & -\RO_g
\end{array}
\right),\enskip
S_{07}=\left(
\begin{array}{c|c}
\RO_h & 0 \\ \hline
0 & -\RO_h
\end{array}
\right),\enskip
S_{08}=\left(
\begin{array}{c|c}
0 & -\Id \\ \hline
\Id & 0
\end{array}
\right).
$
}
\]

Via invariant polynomials, one can then get global differential forms on manifolds $M^{16}$, and prove the following facts, completing some of the statements already proved in \cite{pp}:
\begin{theorem}\label{acs:7->8->9}
(i) The families of complex structures
\[
S^A = \{S_{\alpha \beta}\}_{1 \leq \alpha < \beta \leq 7}, \qquad S^B = \{S_{\alpha \beta}\}_{0 \leq \alpha < \beta \leq 7}, \qquad \text{and} \; \; \;S^C = \{S_{\alpha \beta}\}_{0 \leq \alpha < \beta \leq 8},
\]
provide bases of Lie subalgebras $\mathfrak{spin}_\Delta(7)$, $\liespin{8}$, and $\liespin{9} \subset \mathfrak{so}(16)$, respectively.\\
\noindent(ii) Let
\[
\psi^A=(\psi_{\alpha \beta})_{1 \leq \alpha, \beta \leq 7}, \qquad \psi^B=(\psi_{\alpha \beta})_{0 \leq \alpha, \beta \leq 7}, \qquad 
\psi^C =(\psi_{\alpha \beta})_{0 \leq \alpha, \beta \leq 8}
\]
be the skew-symmetric matrices of the K\"ahler 2-forms associated with the mentioned families of complex structures $S _{\alpha \beta}$. If $\tau_2$ and $\tau_4$ are the second and fourth coefficient of the characteristic polynomial, then:
\[
\frac{1}{6} \tau_2(\psi^A) = \Phi_{\mathrm{Spin}_\Delta(7)}, \quad \frac{1}{4}\tau_2(\psi^B) = \Phi_{\Spin{8}}, \quad  
\tau_2(\psi^C)=0, \quad  \frac{1}{360} \tau_4(\psi^C) = \Phi_{\Spin{9}},
\]
where 
$ \Phi_{\mathrm{Spin}_\Delta(7)} \in \Lambda^4(\RR^{16})$ restricts on any summand of $\RR^{16} =\RR^8 \oplus \RR^8$ to the usual $\Spin{7}$ $4$-form, and where
$ \Phi_{\Spin{9}} \in \Lambda^8(\RR^{16})$ is the canonical form associated with 
the standard $\Spin{9}$-structure in $\RR^{16}$. 
\end{theorem}
The 8-form $\Phi_{\Spin{9}}$ was originally defined by M. Berger in 1972, cf. \cite{be}. 
\begin{proof}
(i) The three families refer to Lie subalgebras of $\liespin{9}$. Now, the family $S^C = \{S_{\alpha \beta}\}_{0 \leq \alpha < \beta \leq 8}$ is known to be a basis of  $\liespin{9}$, cf. \cite{fr,pp}. Look at the construction of the  $S_{\alpha \beta} = S_\alpha S_\beta$,
following the approach to $\Spin{9}$ as generated by transformations of type \eqref{HarSpC} (cf.~\cite[pages 278--279]{ha}). In this construction, matrices in the subalgebra $\liespin{8} \subset \liespin{9}$ are characterized through the infinitesimal triality principle as:
\[
\left(
\begin{array}{cc}
a_+ & 0 \\ 
0& a_-
\end{array}
\right), 
\]
where $a_+,a_- \in \mathfrak{so}(8)$ are \emph{triality companions}, i.e.~for each $u \in \OO$ there exists \mbox{$v=a_0(u)$} such that $R_v = a_+ R_u a_-^t$. It is easily checked that all matrices $S_{\alpha\beta}$ with $0 \leq \alpha < \beta \leq 7$ satisfy this condition. Moreover, matrices in  $\mathfrak{spin}_\Delta(7) \subset \liespin{8}$ are characterized as those with $a_+=a_-$ (thus $a_0= 1)$ (\cite[pages 278--279, 285]{ha}). Thus only the $S_{\alpha \beta}$ with $1 \leq \alpha < \beta \leq 7$ are in $\mathfrak{spin}_\Delta(7)$.\\
\noindent(ii) Here one can write explicit expressions of the $\psi_{\alpha \beta}$ in the coordinates of $\RR^{16}$ (cf.~\cite[pages 334--335]{pp}). These formulas allow to compute the $\tau_2$ and the $\tau_4$ appearing in the statements. It is convenient to begin with the matrix $\psi^B$, by adding up squares of the 2-forms $\psi_{\alpha \beta}$ with $0 \leq \alpha <\beta \leq 7$:

\begin{equation*}%\label{28sums}
\begin{split}
\frac{1}{4} \tau_2(\psi^B)=\frac{1}{4}\sum_{0 \leq \alpha < \beta \leq 7}\psi^2_{\alpha \beta}&=\shortform{121'2'}+\shortform{123'4'}+\shortform{125'6'}-\shortform{127'8'}+\shortform{341'2'}+\shortform{343'4'}-\shortform{345'6'}\\
+\shortform{347'8'}+\shortform{561'2'}-\shortform{563'4'}+\shortform{565'6'}&+\shortform{567'8'}-\shortform{781'2'}+\shortform{783'4'}+\shortform{785'6'}+\shortform{787'8'}+\shortform{131'3'}-\shortform{132'4'}\\
+\shortform{135'7'}+\shortform{136'8'}-\shortform{241'3'}+\shortform{242'4'}&+\shortform{245'7'}+\shortform{246'8'}+\shortform{571'3'}+\shortform{572'4'}+\shortform{575'7'}-\shortform{576'8'}+\shortform{681'3'}\\
+\shortform{682'4'}-\shortform{685'7'}+\shortform{686'8'}+\shortform{141'4'}&+\shortform{142'3'}+\shortform{145'8'}-\shortform{146'7'}+\shortform{231'4'}+\shortform{232'3'}-\shortform{235'8'}+\shortform{236'7'}\\
+\shortform{581'4'}-\shortform{582'3'}+\shortform{585'8'}+\shortform{586'7'}&-\shortform{671'4'}+\shortform{672'3'}+\shortform{675'8'}+\shortform{676'7'}+\shortform{151'5'}-\shortform{152'6'}-\shortform{153'7'}\\
-\shortform{154'8'}-\shortform{261'5'}+\shortform{262'6'}-\shortform{263'7'}&-\shortform{264'8'}-\shortform{371'5'}-\shortform{372'6'}+\shortform{373'7'}-\shortform{374'8'}-\shortform{481'5'}-\shortform{482'6'}\\
-\shortform{483'7'}+\shortform{484'8'}+\shortform{161'6'}+\shortform{162'5'}&-\shortform{163'8'}+\shortform{164'7'}+\shortform{251'6'}+\shortform{252'5'}+\shortform{253'8'}-\shortform{254'7'}-\shortform{381'6'}\\
+\shortform{382'5'}+\shortform{383'8'}+\shortform{384'7'}+\shortform{471'6'}&-\shortform{472'5'}+\shortform{473'8'}+\shortform{474'7'}+\shortform{171'7'}+\shortform{172'8'}+\shortform{173'5'}-\shortform{174'6'}\\
+\shortform{281'7'}+\shortform{282'8'}-\shortform{283'5'}+\shortform{284'6'}&+\shortform{351'7'}-\shortform{352'8'}+\shortform{353'5'}+\shortform{354'6'}-\shortform{461'7'}+\shortform{462'8'}+\shortform{463'5'}\\
+\shortform{464'6'}+\shortform{181'8'}-\shortform{182'7'}+\shortform{183'6'}&+\shortform{184'5'}-\shortform{271'8'}+\shortform{272'7'}+\shortform{273'6'}+\shortform{274'5'}+\shortform{361'8'}+\shortform{362'7'}\\
+\shortform{363'6'}-\shortform{364'5'}+\shortform{451'8'}+\shortform{452'7'}&-\shortform{453'6'}+\shortform{454'5'},
\end{split}
\end{equation*}
and this can be defined as $\Phi_{\mathrm{Spin}(8)}$.
By computing the sum of squares of the $\psi_{\alpha \beta}$ with $1 \leq \alpha < \beta \leq 7$ one gets instead: 
\begin{equation*}%\label{21sums}
\begin{split}
\tau_2(\psi^A)&=\sum_{1 \leq \alpha < \beta \leq 7} \psi^2_{\alpha \beta}=\frac{6}{4}\tau_2(\psi^B)+6[\shortform{1234}+\shortform{5678}+\shortform{1'2'3'4'}+\shortform{5'6'7'8'}]\\
&-3[\shortform{15}+\shortform{26}+\shortform{37}+\shortform{48}]^2-3[\shortform{1'5'}+\shortform{2'6'}+\shortform{3'7'}+\shortform{4'8'}]^2 \\
&-6[\shortform{1278}-\shortform{1368}+\shortform{1467}+\shortform{2358}-\shortform{2457}+\shortform{3456}+\shortform{1'2'7'8'}-\shortform{1'3'6'8'}+\shortform{1'4'6'7'}+\shortform{2'3'5'8'}\\
&\phantom{\,\,\,-6[\shortform{1278}-\shortform{1368}+\shortform{1467}+\shortform{2358}-\shortform{2457}+\shortform{3456}+\shortform{1'2'7'8'}-\shortform{1'3'6'8'}}-\shortform{2'4'5'7'}+\shortform{3'4'5'6'}].
\end{split}
\end{equation*}
Thus, by defining $\Phi_{\mathrm{Spin}_\Delta(7)} = \frac{1}{6} \tau_2(\psi^A)$ one has that its restriction to any of the two summands $\RR^{16}=\RR^8 \oplus \RR^8$ is the usual $\Spin{7}$ form (cf. \cite[pages 332--333]{pp}).
Moreover, computations carried out in \cite[pages 336 and 343]{pp} show that:
\begin{equation*}%\label{36sums}
\tau_2(\psi^C) = \sum_{1 \leq \alpha < \beta \leq 9} \psi^2_{\alpha \beta} = 0, \qquad \Phi_{\Spin{9}}=\frac{1}{360}\tau_4(\psi^C).
\end{equation*}
\end{proof}

The coefficients in the above equalities are chosen in such a way that, when reading 
\[
\Phi_{\mathrm{Spin}_\Delta(7)}, \qquad \Phi_{\Spin{8}}\in \Lambda^4 , \qquad  \Phi_{\Spin{9}} \in \Lambda^8
\]
in the coordinates of $\RR^{16}$, the g.c.d.\ of coefficients be 1. 

Theorem \ref{acs:7->8->9} suggests to consider, besides the Clifford systems $C_5$ and $C_8$ on $\RR^{16}$, the following intermediate Clifford systems:
\[
C_6=(S_1, \dots , S_7), \qquad C_7=(S_0, \dots , S_7),
\]
or, in accordance with Theorem \ref{Procedure}, the equivalent Clifford systems
\[
C'_6=(S_0, S_1, \dots , S_5, S_8), \qquad C'_7=(S_0, S_1, \dots , S_6, S_8).
\]

\begin{remark}\label{Spin7}
Similarly to what observed in Remark \ref{two}, a representative of the other equivalence class of Clifford systems with $m=8$ can be constructed as $\tilde C_8=(S_0, \tilde S_1, \dots , \tilde S_7,S_8)$, where $\tilde S_1, \dots , \tilde S_7$ are defined like $S_1, \dots , S_7$ but using the left octonion multiplications $L_i ,\dots , L_h$ instead of the right ones $R_i,\dots ,R_h$.
\end{remark}

\begin{remark}\label{dh}
A family of calibrations $\phi_{4k}(\lambda) \in \Lambda^{4k}(\RR^{16})$ has been constructed by J. Dadok and F. R. Harvey for any $\lambda$ in the unit sphere $S^7 \subset \RR^8$ by squaring positive spinors $S(\lambda) \in {\bf S^+}(16) $, through the following procedure, cf. \cite{dh}. Write $\RR^{16}$ as ${\OO}^+ \oplus {\OO}^-$, and let $(s_1 =i, \dots , s_8=h)$ be the standard basis of the octonions $\OO^+, \OO^- $. Then one looks at the model for the Clifford algebra
\[
\begin{split}
\Lambda (\RR^{16}) \cong \mathrm{Cl_{16}} \cong \mathrm{Cl_{8}} \otimes \mathrm{Cl_{8}} \cong & \; \mathrm{End} (\OO^+ \oplus \OO^-) \otimes \mathrm{End} (\OO^+ \oplus \OO^-) \cong  \\
\mathrm{End} (\OO^+ \otimes \OO^+) \oplus \mathrm{End} (\OO^- \otimes \OO^-) & \oplus \mathrm{End} (\OO^+ \otimes \OO^-) \oplus \mathrm{End} (\OO^- \otimes \OO^+),
\end{split}
\]
and the spinors $S(\lambda)$ are in the diagonal $D \subset{\bf S^+}(16) =  (\OO^+ \otimes \OO^+) \oplus (\OO^- \otimes \OO^-)$. Then it is proved in \cite{dh} that by squaring such spinors one gets inhomogeneous exterior forms in $\RR^{16}$ as
\[
256 \; S(\lambda) \circ S(\lambda) = 1 + \phi_4 + \phi_8 +\phi_{12} + \, vol ,
\]
where the $\phi_{4k}$ are calibrations. In particular, calibrations corresponding to $\Spin{7}$ and $\Spin{8}$ geometries are determined and discussed in \cite{dh}. This construction can be related with the present point of view in terms of Clifford systems, as we plan to show in a forthcoming work. 
\end{remark}

\section{Clifford systems $C_m$ on Riemannian manifolds}

The definition \ref{defspin9} of a $\Spin{9}$ structure on a Riemannian manifold $M^{16}$, using locally defined Clifford systems $C_8$ on its tangent bundle, and yielding a rank 9 vector subbundle of the endomorphism bundle, suggests to give the following more general definition.

\begin{definition} A \emph{Clifford system $C_m$ on a Riemannian manifold $(M^N,g)$} is the datum of a rank $m+1$ vector subbundle $E^{m+1} \subset \End{TM}$ locally generated by Clifford systems that are related in the intersections of trivializing open sets by matrices in $\SO{m+1}$.
\end{definition}

Some of the former statements, like Propositions \ref{C_2}, \ref{C_3}, \ref{C_4}, \ref{C_5}, and similar properties discussed for $C_6,C_7,C_8$ in Section \ref{R16}, can be interpreted on Riemannian manifolds. One can then recognize that the datum of a Clifford system $C_m$ on a Riemannian manifold $M^N$, $N=2\delta(m)$, is equivalent to the reduction of its structure group to the group $G$ according to the following Table D.
\begin{table}[h]%\label{m}
\caption{Clifford systems $C_m$ and $G$-structures on Riemannian manifolds $M^N$}
\renewcommand{\arraystretch}{1.85}
\begin{center}
%\medskip
\resizebox*{1.00\textwidth}{!}{%
\begin{tabular}{|p{0.10in}  ||c|c|c|c|c|c|c|c|c|c|c|c|c|c|c|}
\hline
$m$&$1$&$2$&$3$&$4$&$5$&$6$&$7$&$8$&$9$&$10$&$11$&$12$\\
\hline 
$N$&$2$&$4$&$8$&$8$&$16$&$16$&$16$&$16$&$32$&$64$&$128$&$128$\\
\hline 
$G$&$\U{1}$&$\U{2}$&$\Sp{1}^3$&$\Sp{2}\Sp{1}$&$\SU{4}\Sp{1}$&$\Spin{7}\U{1}$&$\Spin{8}$&$\Spin{9}$&$\Spin{10}$&$\Spin{11}$&$\Spin{12}$&$\Spin{13}$\\
\hline 
\end{tabular}
}
\end{center}
\end{table}

\begin{remark} Although $\Spin{7}$ structures on 8-dimensional Riemannian manifolds cannot be described through a rank 7 vector bundle of symmetric endomorphisms of the tangent bundle (cf. \cite[Corollary 9]{pp}), this can definitely be done for a $\mathrm{Spin}_\Delta(7)$ structure in dimension 16. The above discussion shows in fact that the 7 symmetric endomorphisms $S_1,\dots,S_7$ allow to deal with a $\mathrm{Spin}_\Delta(7)$ structure as a Clifford system on  a 16-dimensional Riemannian manifold. Indeed, most of the known examples of $\Spin{9}$ manifolds carry the subordinated structure $\mathrm{Spin}_\Delta(7)$, cf. \cite{fr,oppv}. 

On the other hand, a $\Spin{7}$ structure on a 8-dimensional Riemannian manifold is an example of even Clifford structure, as defined in the Introduction. Here the defining vector bundle $E$ has rank 7 and one can choose $E \subset \End{TM}^-$ locally spanned as
\[
<\mathcal I, \mathcal J , \mathcal K , \mathcal E, \mathcal F , \mathcal G , \mathcal H>,
\]
i.e.~multiplying by elements in the canonical basis of octonions. As already mentioned, in situations like this, we call \emph{essential} the even Clifford structure.
\end{remark}

\section{The Clifford system $C_{9}$ and the essential Clifford structure on $\EIII$}\label{sectionEIII}

Our recipe for producing Clifford systems, according to Theorem \ref{Procedure}, gives on $\RR^{32}$ the following Clifford system $C_9=(T_0,T_1,\dots, T_9)$:

\begin{equation*}
\begin{split}
T_0=\left(
\begin{array}{r|r}
0 & \Id \\
\hline 
\Id & 0
\end{array}\right),\dots,
T_\alpha=\left(
\begin{array}{c|c}
0 & -S_{0\alpha} \\
\hline 
S_{0\alpha}& 0
\end{array}\right),\dots,
T_9=\left(
\begin{array}{r|r}
\Id & 0 \\
\hline
0 & -\Id
\end{array}\right).
\end{split}
\end{equation*}
Here $\alpha=1, \dots , 8$ and any block in the matrices is now of order $16$.

In \cite{pp3} we showed that the group of orthogonal transformations preserving the vector subspace $E^{10} =<C_9>  \subset \End{\RR^{32}}$ is the image of $\Spin{10}$ under a real representation in $\SO{32}$. Indeed, one can also look at the half-Spin representations of $\Spin{10}$ into $\SU{16} \subset \SO{32}$, that are related with the notion of \emph{even Clifford structure}, as defined in the Introduction. 

Note that, according to the definitions, any Clifford system $C_m$ on a Riemannian manifold $M^N$ gives rise to an even Clifford structure of rank $m+1$: this is for example the case of $C_4$ on 8-dimensional quaternion Hermitian manifolds, or of $C_8$ on 16-dimensional $\Spin{9}$ manifolds. Indeed, one has also a notion of \emph{parallel} even Clifford structure, requiring the existence of a metric connection $\nabla$ on $(E,h)$ such that $\varphi$ is connection preserving. Thus for example parallel even Clifford structures with $m=4,8$ correspond to a quaternion K\"ahler structure in dimension $8$ and to holonomy $\Spin{9}$ in dimension $16$. In \cite[page 955]{ms}, a classification is given of complete simply connected Riemannian manifolds with a parallel non-flat even Clifford structure. 

This classification statement includes one single example for each value of the rank $m+1=9,10,12,16$ (and no other examples when $m+1>8$). These examples are the ones in the last row (or column) of Table A, namely the projective planes over the four algebras $\OO$, $\CC \otimes \OO$, $\HH \otimes \OO$, $\OO \otimes \OO$. 

In the Cartan labelling they are the symmetric spaces:
\[
\begin{array}{rll}
\FII &= \mathrm{F_4}/ \Spin{9},\qquad &\EIII = \mathrm{E_6}/ \mathrm{Spin}(10) \cdot \mathrm{U}(1),\\
\EVI &= \mathrm{E_7}/ \mathrm{Spin}(12) \cdot \Sp{1},\qquad &\EVIII= \mathrm{E_8}/ \mathrm{Spin}(16)^+.
\end{array}
\]

In this respect we propose the following
\begin{definition} Let $M$ be a Riemannian manifold. An even Clifford structure $(E,h)$, with rank $m+1$ and defining map $\varphi: \; \Cl{0}(E) \rightarrow \End{TM}$, is said to be \emph{essential} if it is not a Clifford system, i.e.~if it is not locally spanned by anti-commuting self-dual involutions.
\end{definition}

We have already seen that $\Spin{7}$ structures in dimension 8 are examples of essential even Clifford structures, cf.~Remark \ref{Spin7}. As mentioned, both quaternion Hermitian structures in dimension $8$ and $\Spin{9}$ structures in dimension $16$ are instead non-essential. For example, on the Cayley plane $\FII$, local Clifford systems on its three coordinate open affine planes $\OO^2$ fit together to define the $\Spin{9}$ structure and hence the even Clifford structure. This property has no analogue for the other three projective planes $\EIII$, $\EVI$, $\EVIII$. As a matter of fact it has been proved in \cite{im} that the projective plane $\EIII$ over complex octonions cannot be covered by three coordinate open affine planes $\CC \otimes \OO^2$.
We have:

\begin{theorem}\label{EIII}
The parallel even Clifford structure on $\EIII$ is essential.
\end{theorem}

\begin{proof} Note first that the statement cannot follow from Proposition \ref{basic}. However, as observed in Table D, the structure group of a $32$-dimensional manifold carrying a Clifford system $C_9$ reduces to $\Spin{10} \subset \SU{16}$. This would be the case of the holonomy group, assuming that such a Clifford system induces the parallel even Clifford structure of $\EIII$. Thus, $\EIII$ would have a trivial canonical bundle, in contradiction with the  positive Ricci curvature property of Hermitian symmetric spaces of compact type.
%This means that one looks here at a half-Spin representation of $\Spin{10}$ in $\SU{16}$, and the complex space $\CC^{16}$ cannot admit $10$ self-dual anti-commuting involutions. Such involutions would give in fact a representation of the complex Clifford algebra $\CC\mathrm{l}_{10} \cong \CC(32)$ on the vector space $\CC^{16}$, cf. \cite{br,pp3}.
\end{proof}

As showed in \cite{pp3}, the vector bundle defining the even Clifford structure is the $E^{10}$ locally spanned as $<\mathcal I> \oplus <S_0, \dots ,S_8>$. Here $\mathcal I$ is the global complex structure of the Hermitian symmetric space $\EIII$, and $S_0, \dots ,S_8$ (matrices in $\SO{16} \subset \SU{16}$) define, together with $\mathcal I$, the $\Spin{10} \cdot \mathrm{U}(1) \subset \U{16}$ structure given by its holonomy. The construction of the even Clifford structure follows of course the alternating composition, so that $\mathcal I \wedge S_\alpha = - S_\alpha \wedge \mathcal I$, allowing to get a skew symmetric matrix $\psi^D=(\psi_{\alpha\beta})$ of K\"ahler forms associated with all compositions of two generators (cf. Theorem \ref{acs:7->8->9}).

\begin{remark}\label{cohomology}
As proved in \cite{pp3}, the cohomology classes of the K\"ahler form and of the global differential form $\tau_4(\psi^D)$ generate the cohomology of $\EIII$.
\end{remark}

\section{Clifford systems $C_{10}$ up to $C_{16}$}\label{other 7}

For the next step, we need now the following order $32$ matrices:
%\begin{footnotesize}
\begin{equation*}%\label{eq:J1}
\begin{split}
T_{01}=\left(
\begin{array}{c|c|c|c}
\RO_i & 0 &0&0\\ \hline
0 & -\RO_i &0&0 \\ \hline
0&0&-\RO_i &0 \\ \hline
0&0&0&\RO_i 
\end{array}
\right),\dots,
T_{07}&=\left(
\begin{array}{c|c|c|c}
\RO_h & 0 &0&0\\ \hline
0 & -\RO_h &0&0 \\ \hline
0&0&-\RO_h &0 \\ \hline
0&0&0&\RO_h
\end{array}
\right),\\
T_{08}=\left(
\begin{array}{c|c|c|c}
0& -\Id &0&0\\ \hline
\Id &0&0&0 \\ \hline
0&0&0&\Id \\ \hline
0&0&-\Id&0
\end{array}
\right),\quad
T_{09}&=\left(
\begin{array}{c|c|c|c}
0 &0 &-\Id &0\\ \hline
0&0&0&-\Id \\ \hline
\Id&0&0&0\\ \hline
0&\Id&0&0 
\end{array}
\right),
\end{split}
\end{equation*}
%\end{footnotesize}
and now we can write on $\RR^{64}$ the following matrices of the Clifford system $C_{10}$:
\begin{equation*}
U_0=\left(
\begin{array}{r|r}
0 & \Id \\
\hline 
\Id & 0
\end{array}\right),\dots,
U_\alpha=\left(
\begin{array}{c|c}
0 & -T_{0\alpha} \\
\hline 
T_{0\alpha}& 0
\end{array}\right),\dots,
U_{10}=\left(
\begin{array}{r|r}
\Id & 0 \\
\hline
0 & -\Id
\end{array}\right).
\end{equation*}
Here $\alpha=1, \dots , 9$ and any block in the matrices is of order $32$.

The subgroup of $\SO{64}$ preserving the subbundle $E^{11} =<C_{10}> \subset \End{\RR^{64}}$ is now $\Spin{11}$, a subgroup of $\Sp{16} \subset \SO{64}$ under its Spin representation.

Still another step, through
\begin{equation*}%\label{eq:J1}
U_{01}=\left(
\begin{array}{c|c}
T_{01}& 0 \\ \hline
0 & -T_{01}
\end{array}
\right),\dots,
U_{09}=\left(
\begin{array}{c|c}
T_{09} & 0\\ \hline
0 & -T_{09} 
\end{array}
\right),
U_{0,10}=\left(
\begin{array}{c|c}
0 &-\Id \\ \hline
\Id&0
\end{array}
\right), 
\end{equation*}
we go to the first Clifford system $C_{11}$ in $\RR^{128}$. Its matrices are:

\begin{equation*}
V'_0=\left(
\begin{array}{r|r}
0 & \Id \\
\hline 
\Id & 0
\end{array}\right),\dots,
V'_\alpha=\left(
\begin{array}{c|c}
0 & -U_{0\alpha} \\
\hline 
U_{0\alpha}& 0
\end{array}\right),\dots,
V'_{11}=\left(
\begin{array}{r|r}
\Id & 0 \\
\hline
0 & -\Id
\end{array}\right),
\end{equation*}
now with $\alpha=1, \dots , 10$ and any block of order $64$.

To recognize the next Clifford system, $C_{12}$ and again in $\RR^{128}$, introduce the following matrices, of order $32$:
\[
\resizebox{\textwidth}{!}{%
$
\text{Block}_{\RH_i} =\left(
\begin{array}{c|c|c|c}
0 & -\Id & 0 & 0 \\ \hline
\Id & 0 & 0 & 0 \\\hline
0 & 0 & 0 & \Id \\\hline
0 & 0 & -\Id & 0
\end{array}
\right),\enskip
\text{Block}_{\RH_j}  =\left(
\begin{array}{c|c|c|c}
0 & 0 & -\Id & 0 \\\hline
0 & 0 & 0 & -\Id \\\hline
\Id & 0 & 0 & 0 \\\hline
0 & \Id & 0 & 0
\end{array}
\right),\enskip
\text{Block}_{\LH_k} =\left(
\begin{array}{c|c|c|c}
0 & 0 & 0 & -\Id \\\hline
0 & 0 & -\Id & 0 \\\hline
0 & \Id & 0 & 0 \\\hline
\Id & 0 & 0 & 0
\end{array}
\right),
$
}
\]
block-wise extensions of matrices $\RH_i,\RH_j,\LH_k$ considered in Section \ref{first four}. We need also the further matrices, of order $64$, block-wise extension of $R_i,R_j,R_e,R_h$:
\begin{equation*}
\begin{split}
\text{Block}_{R_i} =\left(
\begin{array}{c|c}
\text{Block}_{\RH_i}  & 0 \\ \hline
0 & -\text{Block}_{\RH_i}  
\end{array}
\right)&,\qquad
\text{Block}_{R_j} =\left(
\begin{array}{c|c}
\text{Block}_{\RH_j}  & 0 \\ \hline
0 & -\text{Block}_{\RH_j}  
\end{array}
\right),\\
\text{Block}_{R_e} =\left(
\begin{array}{c|c}
0& -\Id \\ \hline
\Id &0
\end{array}
\right)&,\qquad
\text{Block}_{R_h} =\left(
\begin{array}{c|c}
0& \text{Block}_{\LH_k} \\ \hline
\text{Block}_{\LH_k}  & 0
\end{array}
\right).
\end{split}
\end{equation*}

Then one easily writes the last matrices in the Clifford system $C_{11}$ as:
\[
\resizebox{\textwidth}{!}{%
$
V'_8=\left(
\begin{array}{c|c}
0 & -\text{Block}_{R_i} \\ \hline
\text{Block}_{R_i}  & 0 
\end{array}
\right),\enskip 
V'_9=\left(
\begin{array}{c|c}
0& -\text{Block}_{R_j}  \\ \hline
\text{Block}_{R_j} & 0
\end{array}
\right),\enskip 
V'_{10}=\left(
\begin{array}{c|c}
0& -\text{Block}_{R_e} \\ \hline
\text{Block}_{R_e}  &0
\end{array}
\right).
$
}
\]

One gets:
\begin{proposition}
The matrices 
\[
V_0=V'_0, \dots , V_{10}=V'_{10}, V_{11}=\left(
\begin{array}{c|c}
0& -\mathrm{Block}_{R_h}  \\ \hline
\mathrm{Block}_{R_h} & 0
\end{array}
\right) , V_{12} =V'_{11}
\]
give rise to the Clifford system $C_{12}$ in $\RR^{128}$.
\end{proposition}

\begin{proof} 
The only point to check is that $V_{11}$ anti-commutes with all the other matrices. This is a straightforward computation.
\end{proof}

The orthogonal  transformations preserving $C_{12}$ correspond to a real representation of $\Spin{12}$ in $\SO{128}$.

As a further step, we construct $C_{13}$, the first Clifford system in $\RR^{256}$, whose involutions are:
\begin{equation*}
W'_0=
\left(
\begin{array}{r|r}
0 & \Id \\
\hline 
\Id & 0
\end{array}\right),\dots,
W'_\alpha=\left(
\begin{array}{c|c}
0 & -V_{0\alpha} \\
\hline 
V_{0\alpha}& 0
\end{array}\right),\dots,
W'_{13}=\left(
\begin{array}{r|r}
\Id & 0 \\
\hline
0 & -\Id
\end{array}\right),
\end{equation*}
now with $\alpha =1, \dots , 12$. In particular
\[
W'_{12} =\left(
\begin{array}{c|c|c|c}
0 & 0 & 0 & \Id \\\hline
0 & 0 & -\Id & 0 \\\hline
0 & -\Id & 0 & 0 \\\hline
\Id & 0 & 0 & 0
\end{array}
\right)=\left(
\begin{array}{c|c}
0 &-\mathrm{Block}^{128}_{R_e} \\
\hline 
\mathrm{Block}^{128}_{R_e}& 0
\end{array}\right), 
\]
when now the block matrix is or order $128$. Then one recognizes that one can add three similar block matrices with $\mathrm{Block}^{128}_{R_i},\mathrm{Block}^{128}_{R_j},\mathrm{Block}^{128}_{L_h}$, extending $C_{13}$ up to $C_{16}$, still on $\RR^{256}$, and with intermediate Clifford systems $C_{14}$ and $C_{15}$.

\section{The symmetric spaces $\EVI$ and $\EVIII$}\label{essential Clifford structures}

As a consequence of Proposition \ref{basic} we have:

\begin{theorem}\label{EVI}
The parallel even Clifford structures on $\EVI$ and on $\EVIII$ are essential.
\end{theorem}

%\begin{proof}
%The symmetric spaces $\EVI$ and $\EVIII$ have dimension $64$ and $128$, respectively. Their holonomy groups are $\mathrm{Spin}(12) \cdot \Sp{1} \subset \SO{64}$, and $\mathrm{Spin}(16)^+  \subset \SO{128}$. A comparison with the groups listed in Table C shows that their even Clifford structures cannot be described through Clifford systems.

These even Clifford structures can in fact be defined by vector subbundles $E^{12}$ and $E^{16}$ of the endomorphism bundle, locally generated as follows:
\begin{equation}\label{E^12}
E^{12}:<\mathcal I, \mathcal J , \mathcal K > \oplus <S_0, \dots , S_9> \; \longrightarrow \; \EVI,
\end{equation}
\begin{equation}\label{E^16}
E^{16}:<\mathcal I, \mathcal J , \mathcal K , \mathcal E, \mathcal F , \mathcal G , \mathcal H> \oplus <S_0, \dots , S_9> \; \longrightarrow \; \EVIII,
\end{equation}
where $S_0, \dots , S_9$ are the involutions in $\SO{16}$ defining the $\Spin{9}$ structures.
%\end{proof}

Note that the quaternionic structure of $\EVI$, one of the quaternion K\"ahler Wolf spaces, appears as part of its even Clifford structure. 

As already mentioned concerning $\EIII$, also on $\EVI$ and $\EVIII$ the compositions of generators of the even Clifford structure follows the alternating property e.g. $\mathcal I \wedge S_\alpha = - S_\alpha \wedge \mathcal I$. In this way one still has skew-symmetric matrices of K\"ahler forms associated with compositions of two generators (cf. Theorem \ref{acs:7->8->9} and Theorem \ref{EIII}). We denote these skew-symmetric matrices by $\psi^E$ for $\EVI$ and $\psi^F$ for $\EVIII$. One can look at the following sequence of the matrices we introduced: 
\[
\psi^A \subset \psi^B \subset \psi^C \subset \psi^D \subset \psi^E \subset \psi^F,
\]
all producing, via invariant polynomials, global differential forms associated with the associated structure groups
\[
\Spin{7}_\Delta , \; \Spin{8} , \; \Spin{9} , \; \Spin{10}\cdot \U{1} , \; \Spin{12} \cdot \Sp{1}, \; \Spin{16}^+.
\]
(cf.\ also Remark \ref{cohomology}). We can mention here that the (rational) cohomology of $\EVI$ is generated, besides by the class of the quaternion K\"ahler $4$-form, by a $8$-dimensional class and by a $12$-dimensional class. It is thus tempting to represent these classes by $\tau_4(\psi^E)$ and by $\tau_6(\psi^E)$. As for $\EVIII$, it is known that its rational cohomology is generated by classes of dimension $8,12,16,20$. One can also observe, in this last situation of $\EVIII$, that the local K\"ahler forms $\psi_{\alpha\beta}$ associated with the group $\Spin{16}^+$ can be seen for $\alpha < \beta$ in correspondence with a basis of its Lie algebra $\lieso{16}$. As such, they exhaust both families of $36+84=120$ exterior 2-forms appearing in decomposition \ref{decomposition}.

We conclude with two remarks relating the discussed subjects with some of our previous work.

\begin{remark}\label{spheres}
In \cite{pp2} we described a procedure to construct maximal orthonormal systems of tangent vector fields on spheres. For that we essentially used, besides multiplication in $\CC, \HH, \OO$, the $\Spin{9}$ structure of $\RR^{16}$, applied also block-wise in higher dimension. Remind that the maximal number $\sigma(N)$ of linear independent vector field on an odd-dimensional sphere $S^{N-1}$, with $N=(2k+1)2^p16^q$ and $0 \leq p \leq 3$, is given by the Hurwitz-Radon formula
\[
\sigma(N)=2^p+8q-1.
\]
Thus, it does not appear easy to read this number out of Table B, even considering also reducible Clifford systems.

On the other hand, one can recognize from the construction of \cite{pp2} that there is instead a simple relation with even Clifford structures, and that for example the construction of a maximal system of tangent vector fields on spheres $S^{31}, S^{63}, S^{127}$ can be rephrased using the essential even Clifford structures of rank $10,12,16$ on $\RR^{32}, \RR^{64}, \RR^{128}$. Such even Clifford structures exist and are parallel non-flat on the symmetric spaces $\EIII$, $\EVI$, $\EVIII$ (cf.\ proof of Theorem \ref{EIII}, and equations \eqref{E^12}, \eqref{E^16}). Following \cite{pp2}, this point of view can be suitably applied to spheres of any odd dimension.
\end{remark}

\begin{remark}\label{lcp}
In \cite{oppv} we studied the structure of compact locally conformally parallel $\Spin{9}$ manifolds. They are of course examples, together with their K\"ahler, quaternion K\"ahler, and $\Spin{7}$ counterparts, of manifolds equipped with a \emph{locally conformally parallel even Clifford structure}. We can here observe that the following Hopf manifolds
\[
S^{31} \times S^1, \; S^{63}\times S^1, \; S^{127} \times S^1
\]
are further examples of them, with the locally conformally flat metric coming from their universal covering. One can also describe some finite subgroups of $\Spin{10}$, $\Spin{12}$, $\Spin{16}^+$ acting freely on $S^{31}, S^{63}, S^{127}$, respectively, and the list of groups $K$ mentioned in Example 6.6 of \cite{oppv} certainly applies to these three cases. Accordingly, finite quotients like $(S^{N-1}/K) \times S^1$, with $N=32,64,128$, still carry a locally conformally parallel even Clifford structure. Note however that the structure Theorem C proved in \cite{oppv} cannot be reproduced for these higher rank locally conformally parallel even Clifford structures.
\end{remark}

\bigskip

\begin{flushleft}

%~ \textbf{2010 Mathematics Subject Classification. Primary 53C26, 53C27, 53C38.}\\
%~ \textbf{Key words and phrases. Clifford systems, octonions.}\\[2ex]

% Write more than one author separately if they have different 
% affiliations, otherwise write the names on the same line, separeted 
% by commas.
%

Maurizio Parton\\
Dipartimento di Economia, Universit\`a di Chieti-Pescara\\
Viale della Pineta 4, I-65129 Pescara, Italy\\
E-mail address: \texttt{parton@unich.it}\\[2ex]
Paolo Piccinni\\
Dipartimento di Matematica, Sapienza-Universit\`a di Roma\\
Piazzale Aldo Moro 2, I-00185, Roma, Italy\\
E-mail address: \texttt{piccinni@mat.uniroma1.it}\\[2ex]
Victor Vuletescu\\
Faculty of Mathematics and Informatics, University of Bucharest\\ 
14 Academiei str., 70109, Bucharest, Romania\\
E-mail address: \texttt{vuli@fmi.unibuc.ro}\\[2ex]

\end{flushleft}


\footnotesize
\begin{thebibliography}{99}
%\bibitem {Abe} K. Abe, \emph{Closed Regular Curves and the fundamental form on the Projective Spaces}, Proc. Japan Acad. {\bf 68} (1992), 123-125.
%%\bibitem {am} K. Abe, M. Matsubara, \emph{Invariant forms on the exceptional Symmetric Spaces FII and EIII}, Proc. 1996 Korea-Japan Conf. on Transformation Group theory, {\bf 3} (1997),  1-15.
%\bibitem {amerratum} K. Abe, M. Matsubara, \emph{Erratum} to \cite{am}, Private communication by K.~Abe (2009).
%\bibitem{AgrSLN}
%\bibitem{ad} J. F. Adams, Lectures on Exceptional Lie Groups, edited by Z. Mahmud and M. Mimura, Chicago Lectures in Mathematics, 1996.
%\bibitem {al} D. V. Alekseevskij, \emph{Riemannian Spaces with exceptional Holonomy Groups}, Funct. Anal. Prilozh. {\bf 2} (1968), 97-105.
%\bibitem {al-ma} D. V. Alekseevskij, S. Marchiafava, \emph{Quaternionic Structures on a Manifold and subordinated Structures}, Annali di Mat. Pura e Appl.  {\bf 171} (1996), 205-273.
\bibitem {ab} M. Atiyah, J. Berndt, \emph{Projective Planes, Severi Varieties and Spheres}, Papers in honour of Calabi, Lawson, Siu and Uhlenbeck, Surveys in Differential Geometry, vol. VIII, Int. Press, 2003.
%\bibitem{baez} J. C. Baez, \emph{The Octonions},  Bull. Amer. Math. Soc. {\bf 39}  (2002),  145--205;  {\bf 42}  (2005),  213.
%\bibitem{baez2} J. C. Baez, \emph{This Week's Finds in Mathematical Physics: weeks 64 and 106}, \\ http://www.math.ucr.edu/home/baez/twfshort.html
\bibitem{be} M. Berger, \emph{Du c\^ot\'e de chez Pu}, Ann. Sci \'Ecole Norm. Sup. {\bf 5} (1972), 1-44.
%\bibitem {besse1} A. Besse, Manifolds all of whose geodesics are closed, Springer-Verlag, 1978.
%\bibitem {besse2} A. Besse, Einstein Manifolds, Springer-Verlag, 1987.
%\bibitem{borel} A. Borel,  \emph{Sur la cohomologie des espaces fibr\'es principaux et des espaces homog\`enes de groupes de Lie compacts}, Ann. of Math., {\bf 57} (1953), 115-207.
%\bibitem{bh} A. Borel, F. Hirzebruch, \emph{Characteristic Classes and Homogeneous Spaces, I}, Am. J. Math., {\bf 80} (1958), 458-538.
%\bibitem{br-gr} R. B. Brown, A. Gray, \emph{Riemannian Manifolds with Holonomy ${\mathrm Spin}(9)$}, Diff. Geometry in honour of K. Yano, Kinokuniya, Tokyo (1972), 41-59.
\bibitem{br} R. L. Bryant, \emph{Remarks on Spinors in Low Dimensions}, 1999, http://www.math.duke.edu/~bryant/Spinors.pdf
%\bibitem{cgm} M. Castrillon Lopez, P. M. Gadea, I. V. Mykytyuk, \emph{The canonical $8$-form on Manifolds with Holonomy Group $\Spin{9}$}, 
%Int. J. of Geometric Methods in Modern Physics, {\bf 7} (2010), 1159-1183.
%\bibitem{dnw} C. Devchand, J. Nuyts, G. Weingart, \emph{Matryoshka of Special Democratic Forms}, 
%Commun. Math. Ph.ys, {\bf 293} (2010), 545-562.
%\bibitem{co} K. Corlette, \emph{Archimedean Superrigidity and Hyperbolic Geometry}, Ann. of  Math. {\bf
%135} (1992), 165-182.
\bibitem{co} S. Console, C. Olmos, \emph{Clifford systems, algebraically constant second fundamental form and isoparametric hypersurfaces}, Manuscripta Math. 97 (1998), 335-342.
\bibitem{dh} J. Dadok, R. Harvey, \emph{Calibrations and Spinors}, Acta Math. {\bf
170} (1993), 83-120.
%\bibitem{dhm} J. Dadok, R. Harvey, F. Morgan, \emph{Calibrations on $\RR^8$}, Trans. Am. Math. Soc. {\bf
%307} (1988), 1-40.
%\bibitem{dz} H. Dua, X. Zhao, \emph{The Chow rings of generalized Grassmannians}, Found. Comput. Math., {\bf 10} (2010), 245-274.
%\bibitem{e} J. -H. Eschenburg, \emph{Riemannian Geometry and Linear Algebra} and \emph{Symmetric Spaces and Division Algebras} (2012), http://www.math.uni-augsburg.de/~eschenbu/ 
\bibitem{fkm} D. Ferus, H. Karcher, H. F. M\"unzner, \emph{Cliffordalgebren und neue isoparametrische Hyperfl\"achen},  
Math. Z. {\bf 177}  (1981),  479-502. 
\bibitem{fr} Th. Friedrich, \emph{Weak $\Spin{9}$-Structures on 16-dimensional Riemannian Manifolds},  
Asian J. Math. {\bf 5}  (2001),  129--160. 
\bibitem{gr} C. Gorodski, M. Radeschi, \emph{On homogeneous composed Clifford foliations},  arXiv:1503.09058v1 (2015). 
%\bibitem{fr2} Th. Friedrich, \emph{$\Spin{9}$-Structures and Connections with totally skew-symmetric Torsion},  
%J. Geom. Phys. {\bf 47}  (2003),  197-206. 
%\bibitem{gl-wa-zi} H. Gluck, F. Warner, W. Ziller, \emph{The geometry of the Hopf fibrations}, L'Enseignement Math. {\bf 32}
%(1986), 173-198.
\bibitem{ha} R. Harvey, Spinors and Calibrations, Academic Press, 1990.
\bibitem{hu} D. Husemoller, Fibre Bundles, 3rd ed. Springer, 1994.
%\bibitem{ha-la} R. Harvey, H. B. Lawson Jr., \emph{Calibrated Geometries}, Acta Math. {\bf 148} (1982), 47-157. 
%\bibitem{hs} F. Hirzebruch, P. Slodowy, \emph{Elliptic genera, involutions, and homogeneous spin manifolds}, Geom. Dedicata {\bf 35} (1990), 309-343; Appendix by J. G. Bliss, R. V. Moody, A. Pianzola, 
%Geom. Dedicata {\bf 35} (1990), 345-351.
%\bibitem{hu} D. Husemoller., \emph{Fibre Bundles}, 2nd edition, Springer-Verlag (1975). 
%\bibitem{ipp} S. Ivanov, M. Parton, P. Piccinni, \emph{Locally conformal parallel $\Gtwo$and $\Spin{7}$-Structures}, Math. Res. Letters {\bf 13} (2006), 1001-1011.
%\bibitem{J1} D. D. Joyce, \emph{Compact Manifolds with Special Holonomy}, Oxford University Press, 2000.
\bibitem{im} A. Iliev,  L. Manivel, \emph{The Chow ring of the Cayley plane}. Compositio Math. {\bf 141} (2005), 146-160.
%\bibitem{i} S. Ishihara, \emph{Quaternion K\"ahler manifolds}. J. Diff. Geom. {\bf 9} (1974), 483-500.
%\bibitem{it} K. Ishitoya, H. Toda, \emph{On the cohomology of irreducible symmetric spaces of exceptional type}, J. Math. Kyoto Univ., {\bf 17} (1977), 225-243.
%\bibitem{lm} J. M. Landsberg,  L. Manivel, \emph{The projective geometry of Freudenthal's magic square}. J. Algebra {\bf 239} (2001), 477?512.
\bibitem{lm} H. B. Lawson - M.-L. Michelson, Spin Geometry. Princeton Univ. Press, 1989.
%\bibitem{J2} D. D. Joyce, \emph{Riemannian Holonomy Groups and Calibrated Geometry}, Oxford University Press, 2007.
%\bibitem{KoNFD2} S.~Kobayashi, K.~Nomizu, Foundations of differential geometry, volume~II, Interscience Publishers, 1969.
%\bibitem{mp} A. Moroianu,  M. Pilka, \emph{
%Higher rank homogeneous Clifford structures},
%J. Lond. Math. Soc., {\bf 87} (2013), 384-400.
\bibitem{ms} A. Moroianu,  U. Semmelmann, \emph{
Clifford structures on Riemannian manifolds}, Adv. Math. {\bf 228} (2011), 940-967.
%\bibitem{n} M. Nakagawa, \emph{
%The mod. 2 cohomology ring of the symmetric space $\mathrm{E} \; VI$},
%J. Math. Kyoto Univ., {\bf 41} (2001), 535-557.
%\bibitem{n2} M. Nakagawa, \emph{
%The integral cohomology ring of $\mathrm{E}_8/T$},
%Proc. Japan Acad., {\bf 86} (2010), 64-68 .
%\bibitem{og} A. A.Ognikyan, \emph{Combinatorial Construction of Tangent Vector Fields on Spheres}, Math. Notes {\bf 83} (2008), 590-605.
%\bibitem{op} L. Ornea, P. Piccinni, \emph{Locally conformal K\"ahler Structures in Quaternionic Geometry}, Trans. Am. Marth. Soc., {\bf 349} (1997), 641-655.
%\bibitem{op3} L. Ornea, P. Piccinni, \emph{Cayley 4-frames and a quaternion K\"ahler reduction related to
%$\Spin{7}$}, in Global Diff. Geom.: The Mathematical Legacy of Alfred Gray, Contemp. Math. {\bf 288} (2002), 401-405.
%\bibitem{pa1} M. Parton, \emph{Old and new Structures on Products of Spheres}, in Global Diff. Geom.: The Mathematical Legacy of Alfred Gray, Contemp. Math. {\bf 288} (2002), 406-410.
%\bibitem{mathurl} M. Parton, \emph{${\mathcal Mathematica}$ Computations on $\Spin{9}$}, http://www.sci.unich.it/~parton/index.php?t=rice
\bibitem{oppv} L. Ornea, M. Parton, P. Piccinni, V. Vuletescu, \emph{$\Spin{9}$ Geometry of the Octonionic Hopf Fibration}, Transf. Groups, {\bf 18} (2013), 845-864.
\bibitem{pp} M.Parton, P. Piccinni, \emph{$\Spin{9}$ and almost complex structures on 16-dimensional
  manifolds}, Ann. Gl. Anal. Geom., {\bf 41} (2012), 321--345.
\bibitem{pp2} M. Parton, P. Piccinni, \emph{Spheres with more than 7 vector fields: All the fault of $\Spin{9}$}, Lin. Algebra and its Appl., {\bf 438} (2013), 113-131.
\bibitem{pp3} M. Parton, P. Piccinni, \emph{The even Clifford structure of the fourth Severi variety}, Complex Manifolds, {\bf 2} (2015), Topical Issue on Complex Geometry and Lie Groups, 89-104.
\bibitem{r} M. Radeschi, \emph{Clifford algebras and newsingular Riemannian foliations in spheres},  
Geo. Funct. Anal. {\bf 24}  (2014),  515-559. 
%\bibitem{ro} B. Rosenfeld, Geometry of Lie Groups, Kluwer Ac. Publ., 1997.
%\bibitem{ds}  D. A. Salamon, Th. Walpuski, \emph{Notes on the Octonians}, arXiv:1005.2820 (2010), 1-73.
%\bibitem{Sa}  S.  M. Salamon, Riemannian Geometry and Holonomy Groups, Longman Sc. and Tech., 1989.
\bibitem{tr} A. Trautman, \emph{Clifford Algebras and their Representations}, Encycl. of Math. Physics, vol. 1, ed. J.-P. Francoise et al., Oxford: Elsevier 2006, pp. 518-530.
%\bibitem{seg} B. Segre, Prodromi di Geometria Algebrica, Cremonese, 1972.
%\bibitem{se}  F. Severi, \emph{Intorno ai punti doppi impropri di una superficie generale dello spazio a quattro dimensioni, e ai suoi punti tripli apparenti}, Rend. Circ. Mat. Palermo {\bf 15} (1901), 33-51.
%\bibitem{slo} P. Slodowy, \emph{On the signatures of homogeneous manifolds}, Geom. Dedicata {\bf 43} (1992), 109-120.
%\bibitem{tw} H. Toda, T. Watanabe, \emph{The integral cohomology ring of $\mathrm{F}_4/T$ and $\mathrm{E}_6/T$},
%J. Math. Kyoto Univ., {\bf 14} (1974), 257-286.
%\bibitem{t} B. Totaro, \emph{The torsion index of $\mathrm{E}_8$ and other groups}, Duke Math. J. {\bf 129} (2005), 219-248.
%\bibitem{yo} I. Yokota, \emph{Exceptional Lie Groups},  arXiv:0902.0431 (2009), 204 pp.
\bibitem{zak} F. L. Zak, \emph{Severi varieties}, Math USSR Sbornik, {\bf 54} (1986), 113-127.
\end{thebibliography}
\end{document}